\documentclass{amsart}

\usepackage{amsthm, amssymb, amsmath}
\usepackage[english]{babel}
\usepackage{graphicx}
\usepackage{subfigure}

\begin{document}

\title{Two versions of a specific natural extension}
\author{Karma Dajani}
\address{Department of Mathematics\\
Utrecht University\\
Postbus 80.000\\
3508 TA Utrecht\\
the Netherlands} \email{dajani@math.uu.nl}
\author{Charlene Kalle}
\address{Department of Mathematics\\
Utrecht University\\
Postbus 80.000\\
3508 TA Utrecht\\
the Netherlands} \email{kalle@math.uu.nl}

\maketitle

\begin{abstract}
We give two versions of the natural extension of a specific greedy $\beta$-transformation with deleted digits. We use the natural extension to obtain an explicit expression for the invariant measure, equivalent to the Lebesgue measure, of this $\beta$-transformation.
\end{abstract}

\newtheorem{prop}{Proposition}[section]
\newtheorem{theorem}{Theorem}[section]
\newtheorem{lemma}{Lemma}[section]
\newtheorem{cor}{Corollary}[section]
\newtheorem{remark}{Remark}[section]
\theoremstyle{definition}
\newtheorem{defn}{Definition}[section]
\newtheorem{ex}{Example}[section]
\newcommand{\lex}{<_{\text{lex}}}
\newcommand{\T}{\mathcal T}

\section{Introduction}
The classical greedy $\beta$-transformation, $T_c$, is defined for each real number $\beta >1$ and has been studied by a large number of people. It is defined from the interval $[ 0, \frac{\lfloor \beta \rfloor}{\beta -1}]$ to itself and the definition is as follows.
$$T_c x = \left\{
\begin{array}{ll}
\beta x -j, & \text{if } x \in [\frac{j}{\beta}, \frac{j+1}{\beta}), \; j \in \{ 0,1, \ldots, \lfloor \beta \rfloor -1 \},\\
\beta x - \lfloor \beta \rfloor, & \text{if } x \in [ \frac{\lfloor \beta \rfloor}{\beta}, \frac{\lfloor \beta \rfloor}{\beta -1}],
\end{array}
\right.
$$
where $\lfloor \beta \rfloor$ indicates the largest integer less than or equal to $\beta$. The importance of this transformation lies in the fact that it can be used to generate $\beta$-expansions for all elements in the interval $[ 0, \frac{\lfloor \beta \rfloor}{\beta -1}]$ in the following way. Let $x \in [ 0, \frac{\lfloor \beta \rfloor}{\beta -1}]$ and define the sequence of digits $\{ b_n \}_{n \ge 1}$ by setting
$$ b_1 = b_1(x) = \left\{
\begin{array}{ll}
j, & \text{if } x \in [\frac{j}{\beta}, \frac{j+1}{\beta}), \; j \in \{ 0,1, \ldots, \lfloor \beta \rfloor -1 \},\\
\lfloor \beta \rfloor, & \text{if } x \in [ \frac{\lfloor \beta \rfloor}{\beta}, \frac{\lfloor \beta \rfloor}{\beta -1}],
\end{array}
\right.
$$
and for $n \ge 1$, set $b_n = b_n(x) = b_1(T_c^{n-1}x)$. Then $T_c x = \beta x - b_1$ and inverting this relation gives $x = \frac{b_1}{\beta} + \frac{T_c x}{\beta}$. Repeating this $n$ times leads to $ x = \sum_{i=1}^{n} \frac{b_i}{\beta^i} + \frac{T_c^n x}{\beta^n}$ and for $n \to \infty$, this converges to
$$x= \sum_{i=1}^{\infty} \frac{b_i}{\beta^i}.$$
This last expression is called a $\beta$-expansion of $x$ with digits in the set $\{ 0,1, \ldots, \lfloor \beta \rfloor \}$. More specifically, the expansion obtained by iterating the transformation $T_c$ is called the greedy $\beta$-expansion of $x$, since for each $n \ge 1$, if $b_1, \ldots, b_{n-1}$ are known, then $b_n$ is the largest element of the set $\{ 0, 1, \ldots, \lfloor \beta \rfloor \}$, such that
$$\sum_{i=1}^{n} \frac{b_i}{\beta^i} \le x.$$
There exists an invariant measure for $T_c$, that is absolutely continuous with respect to the Lebesgue measure. From now on, we will call such a measure an {\it acim} and we will use $\lambda$ to denote the 1-dimensional Lebesgue measure. The acim for $T_c$ has the interval $[0,1)$ as its support. In 1957 R\'enyi proved the existence of such a measure (\cite{Ren1}) and in 1959 and 1960 Gel'fond and Parry gave, independently of one another, an explicit expression of the density of this measure (see \cite{Gel1} and \cite{Par1}). This density function $h_c$ is given by
$$
h_c: [0,1) \to [0,1) : x \mapsto \frac{1}{F(\beta)} \sum_{n=0}^{\infty} \frac{1}{\beta^n} 1_{[0, T_{c}^n 1)}(x),
$$
where $F(\beta) =  \int_0^1 \sum_{x < T_{c}^n 1} \frac{1}{\beta^n} d\lambda$ is a normalizing constant.\\
\indent The greedy $\beta$-transformation with deleted digits is a generalization of the classical greedy $\beta$-transformation. For each $\beta >1$ and each set of real numbers $A= \{ a_0, a_1, \ldots, a_m\}$ satisfying
\begin{itemize}
\item[(i)] $a_0=0$,
\item[(ii)] $a_0 < a_1 < \ldots < a_m$,
\item[(iii)] $\max_{1\le j \le m} (a_j - a_{j-1}) \le \frac{a_m}{\beta -1}$,
\end{itemize}
the greedy $\beta$-transformation with deleted digits is defined from the interval $[ 0, \frac{a_m}{\beta-1} ]$ to itself by
$$ T_{dd} \; x = \left\{
\begin{array}{ll}
\beta x - a_j, & \text{if } x \in [ \frac{a_j}{\beta}, \frac{a_{j+1}}{\beta} ], \; j \in \{0, \ldots, m-1\},\\
\beta x - a_m, & \text{if } x \in [ \frac{a_m}{\beta}, \frac{a_m}{\beta-1} ].
\end{array}
\right.
$$
Notice that we get $T_c$ by taking $A=\{ 0,1, \ldots , \lfloor \beta \rfloor \}$. The transformation was first defined in \cite{DK2} and its definition was based on a recursive algorithm given by Pedicini in \cite{Ped1}. In \cite{DK2} the greedy $\beta$-transformations with deleted digits are also defined for digit sets $A$, not satisfying $a_0=0$, but it is shown in the same paper that these transformations are isomorphic to the one given above. So without loss of generality we can assume that $a_0=0$. The transformation $T_{dd}$ can be used to generate $\beta$-expansions with digits in the set $A$ for all elements in the interval $[ 0, \frac{a_m}{\beta-1} ]$ in exactly the same way as described above for the classical transformation. For $x \in [ 0, \frac{a_m}{\beta-1} ]$, set 
$$d_1 = d_1(x) = \left\{
\begin{array}{ll}
a_j, & \text{if } x \in [ \frac{a_j}{\beta}, \frac{a_{j+1}}{\beta} ], \; j \in \{0, \ldots, m-1\},\\
a_m, & \text{if } x \in [ \frac{a_m}{\beta}, \frac{a_m}{\beta-1} ],
\end{array}
\right.$$
and for $n \ge 1$, set $d_n = d_n (x) = d_1 (T_{dd}^{n-1})$. Then $T_{dd} x =\beta x -d_1$ and for each $x \in [ 0, \frac{a_m}{\beta-1} ]$ we can form the expression
\begin{equation}\label{q:greedyexpdd}
x = \sum_{n=1}^{\infty} \frac{d_n}{\beta^n}.
\end{equation}
Expression (\ref{q:greedyexpdd}) is called the greedy $\beta$-expansion with deleted digits of $x$. This expansion is called greedy for the same reasons as before. At each step the digit given by $T_{dd}$ is the largest element of the set $A$ that ``fits in that position of the expansion'', i.e. if $d_1, \ldots, d_{n-1}$ are already known, then $d_n$ is the largest element of $A$, such that
$$ \sum_{i=1}^{n} \frac{d_n}{\beta^n} \le x.$$
Pedicini studied $\beta$-expansions with deleted digits in \cite{Ped1}.\\
\indent In \cite{DK1} it is shown that the transformation $T_{dd}$ admits an acim that is unique and ergodic. The support of this invariant measure is an interval of the form $[0, a_{j_0}-a_{j_0-1} )$, where
$$ j_0 = \min \{j: T_{dd}[0,a_j-a_{j-1}) \subseteq [0, a_j-a_{j-1}) \; \lambda \text{ a.e. }, 1 \le j \le m \}.$$
An explicit expression for the density of this measure, however, is given only under certain conditions. In this paper we will construct two versions of the natural extension of the dynamical system 
$$([0, a_{j_0}-a_{j_0-1} ), \mathcal B([0, a_{j_0}-a_{j_0-1} )), \mu, T),$$
where $ \mathcal B([0, a_{j_0}-a_{j_0-1} ))$ is the Borel $\sigma$-algebra on $[0, a_{j_0}-a_{j_0-1} )$, $T$ is the specific greedy $\beta$-transformation with deleted digits that will be defined below, and $\mu$ is the probability measure on $([0, a_{j_0}-a_{j_0-1} ), \mathcal B([0, a_{j_0}-a_{j_0-1} ))$, obtained by ``pulling back'' the invariant measure that we will define on the natural extension. Notice that the dynamical system $([0, a_{j_0}-a_{j_0-1} ), \mathcal B([0, a_{j_0}-a_{j_0-1} )), \mu, T)$ is not invertible. The natural extension is the smallest invertible dynamical system, that contains this system. The original system can be obtained from the natural extension through a surjective, measurable and measure preserving map that preserves the dynamics of both systems. This map is called a factor map and in this paper it will simply be the projection onto the first coordinate. For more information on natural extensions, see \cite{Roh1} or \cite{Cor1}. By defining the right measure on the natural extension, we can obtain an expression for the density function of the invariant measure of the specific transformation $T$. Maybe one of the versions given in this paper can serve as a starting point for finding an explicit expression for the invariant measure of the greedy $\beta$-transformations with deleted digits in general.\\
\indent The transformation we will consider is the greedy $\beta$-transformation with deleted digits with $\beta = \frac{1+\sqrt 5}{2}$, the positive solution to the equation $x^2-x-1=0$, and with digit set $A = \{ 0,2,3\}$. The support of the acim is the interval $[0,2)$ and therefore we will define the transformation on this interval only. Let the partition $\Delta = \{ \Delta(0), \Delta(2), \Delta(3) \}$ of the interval $[0,2)$ be given by
$$ \Delta(0) = \left[ 0, \frac{2}{\beta} \right), \quad \Delta(2) = \left[ \frac{2}{\beta}, \frac{3}{\beta} \right), \quad \Delta(3) = \left[ \frac{3}{\beta}, 2 \right).$$ 
Then $T: [0,2) \to [0,2)$ is defined by $Tx = \beta x - j$ on $\Delta(j)$, $j \in \{0,2,3\}$. We will use the first section of this paper to fix some notation. In the second and third sections we define two versions of the natural extension of the dynamical system $([0,2), \mathcal B ([0,2)), \mu, T)$. For the classical greedy $\beta$-transformation, versions of the natural extension are given in \cite{DKS} and by Brown and Yin in \cite{Bro1}. The first version we will give is a generalization of the natural extension defined in \cite{DKS}. The second version is defined on a subset of $\mathbb R^2$ and uses the transformation from the first version. We end the paper with a concluding remark.

\section{Expansions and fundamental intervals}

The transformation $T: [0,2) \to [0,2)$ is defined by setting $Tx = \beta x - j$ on $\Delta(j)$, $j \in \{0,2,3\}$. We can use this transformation to generate expansions of all points in the interval $[0,2)$, with base $\beta$ and digits in the set $\{0,2,3\}$ as was described in the introduction. So for all $x \in [0,2)$ we have the expression (\ref{q:greedyexpdd}). We also write $x =_{\beta} d_1 d_2 d_3 \ldots$, which is understood to mean the same as (\ref{q:greedyexpdd}). Two expansions that will play an important role in what follows are the expansions of the points 1 and $\frac{1}{\beta^3}$. Notice that $\frac{1}{\beta^3} = 2\beta - 3$ would be the image of 2 under $T$ if $T$ were defined on the closed interval $[0,2]$. We have
\begin{eqnarray}
1 &=& \sum_{n=1}^{\infty} \frac{d_n^{(2)}}{\beta^n} = \frac{2}{\beta^2} + \frac{2}{\beta^5} + \frac{2}{\beta^8} + \frac{2}{\beta^{11}} + \ldots =_{\beta} 02\overline{002}, \label{q:exp1}\\
\frac{1}{\beta^3} &=& \sum_{n=1}^{\infty} \frac{d_n^{(3)}}{\beta^n} =  \frac{2}{\beta^4} + \frac{2}{\beta^7} + \frac{2}{\beta^{10}} + \ldots =_{\beta} 00\overline{002},\label{q:exp2}
\end{eqnarray}
where the bars on the right hand side of the previous equations indicate a repeating sequence in the expansions. With the {\it orbit of a point} $x$ {\it under} $T$ we mean the set $ \{ T^n x : n \ge 0\}$. In Figure \ref{f:goldenattractor}, you can see the graph of $T$ and the orbits of the points 1 and $\frac{1}{\beta^3}$.
\begin{figure}[h]
\centering
{\includegraphics[height=4cm]{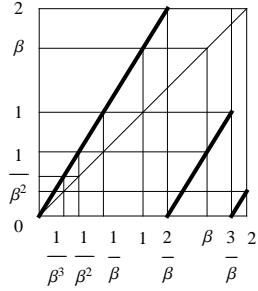}}
\caption{The transformation $T$ and the orbits of $1$ and $\frac{1}{\beta^3}$.}
\label{f:goldenattractor}
\end{figure}

Using $T$ and $\Delta$, we can define a sequence of partitions $\{ \Delta^{(n)} \}_{n \ge 1}$ of $[0,2)$ by setting $ \Delta^{(n)} = \bigvee_{i=0}^{n-1} T^{-i} \Delta$.
We call the elements of $\Delta^{(n)}$ {\it fundamental intervals of rank} $n$. Since they will have the form
$$ \Delta(b_0) \cap T^{-1} \Delta(b_1) \cap \ldots \cap T^{-(n-1)} \Delta(b_{n-1})$$
for some $b_0, b_1, \ldots, b_{n-1} \in \{0,2,3 \}$, we will denote them by $\Delta(b_0 \ldots b_{n-1})$. We will call $\Delta(b_0 \ldots b_{n-1}) \in  \Delta^{(n)}$ {\it full} if $T^n \Delta(b_0 \ldots b_{n-1}) = [0,2)$ and {\it non-full} otherwise. Notice that a fundamental interval of rank $n$ specifies the first $n$ digits, $d_1, \ldots, d_n$, of the greedy expansion of the elements it contains. So,
$$ \Delta(b_0 \ldots b_{n-1}) = \{ x \in [0,2) : d_i(x) = b_{i-1}, \; 1 \le i \le n \}.$$

For full fundamental intervals, we have the following obvious lemma.
\begin{lemma}\label{l:full}
Let $\Delta(a_0 \ldots a_{p-1})$ and $\Delta(b_0 \ldots b_{q-1})$ be two full fundamental intervals of rank $p$ and $q$ respectively. Then the set $\Delta(a_0 \ldots a_{p-1} b_0 \ldots b_{q-1})$
is a full fundamental interval of rank $p+q$.
\end{lemma}

From the next lemma, it follows that the full fundamental intervals generate the Borel $\sigma$-algebra on $[0,2)$.
\begin{lemma}\label{l:generate}
For each $n \ge 1$, let $D_n$ be the union of those full fundamental intervals of rank $n$ that are not subsets of any full fundamental interval of lower rank. Then
$$ \sum_{n=1}^{\infty} \lambda(D_n) =2.$$
\end{lemma}

\begin{proof}
Notice that
$$ \lambda (D_1) = \lambda (\Delta(0)) = \frac{2}{\beta}, \quad \lambda (D_3) = \lambda (\Delta(200)) = \frac{2}{\beta^3}$$
and for $k \ge 1$,
$$ \lambda (D_{3(k+1)}) = \lambda (\Delta(202\underbrace{002\ldots 002}_{k-1 \text{ times}}000) \cup \Delta(300\underbrace{002\ldots 002}_{k-1 \text{ times}}000)) = \frac{4}{\beta^{3(k+1)}}.$$
For all the other values of $n$, $D_n = \emptyset$. So
$$
\sum_{n=1}^{\infty} \lambda(D_n) = \frac{2}{\beta} + \frac{2}{\beta^3} + \sum_{k=1}^{\infty} \frac{4}{\beta^{3(k+1)}} =  \frac{2}{\beta} + \frac{2}{\beta^3} + \frac{4}{\beta^3} \left[ \frac{1}{1-1/\beta^3}-1 \right] =2.
\qedhere $$
\end{proof}

\begin{remark}\label{r:generate}
{\rm The fact that $\Delta(0)$ is a full fundamental interval of rank 1 allows us to construct full fundamental intervals of arbitrary small Lebesgue measure. This together with the previous lemma guarantees that we can write each interval in $[0,2)$ as a countable union of full fundamental intervals. Thus, the full fundamental intervals generate the Borel $\sigma$-algebra on $[0,2)$.
}
\end{remark}

\section{Two rows of rectangles}
To find an expression for the acim of $T$, we will define two versions of the natural extension of the dynamical system $([0,2), \mathcal B([0,2)), \mu, T)$. For the definition of the first version, we will use a subcollection of the collection of fundamental intervals. For $n \ge 1$, let $B_n$ denote the collection of all non-full fundamental intervals of rank $n$ that are not a subset of any full fundamental interval of lower rank. The elements of $B_n$ can be explicitly given as follows.
$$ B_1 = \{ \Delta(2), \Delta(3) \}, \quad B_2 = \{ \Delta(20), \Delta(30)\}$$
and for $k \ge 1$,
\begin{eqnarray*}
B_{3k} &=& \{ \Delta (202\underbrace{002 \ldots 002}_{k-1 \text{ times}}), \Delta (300\underbrace{002 \ldots 002}_{k-1 \text{ times}}) \},\\
B_{3k+1} &=& \{ \Delta (202\underbrace{002 \ldots 002}_{k-1 \text{ times}}0), \Delta (300\underbrace{002 \ldots 002}_{k-1 \text{ times}}0) \},\\
B_{3k+2} &=& \{ \Delta (202\underbrace{002 \ldots 002}_{k-1 \text{ times}}00), \Delta (300\underbrace{002 \ldots 002}_{k-1 \text{ times}}00) \}.
\end{eqnarray*}
Then
$$ T \Delta(2) = [0,1), \; T\Delta(3) = [0, 1/ \beta^3 ), \; T^2 \Delta(20) = [0,\beta), \; T^2\Delta(30) = [0, 1/ \beta^2)$$
and for $k \ge 1$
\begin{eqnarray*}
T^{3k} \Delta (202\underbrace{002 \ldots 002}_{k-1 \text{ times}}) &=& T^{3k} \Delta (300\underbrace{002 \ldots 002}_{k-1 \text{ times}}) = [0, 1/ \beta),\\
T^{3k+1} \Delta (202\underbrace{002 \ldots 002}_{k-1 \text{ times}}0) &=& T^{3k+1} \Delta (300\underbrace{002 \ldots 002}_{k-1 \text{ times}}0) = [0,1),\\
T^{3k+2} \Delta (202\underbrace{002 \ldots 002}_{k-1 \text{ times}}00) &=&T^{3k+2} \Delta (300\underbrace{002 \ldots 002}_{k-1 \text{ times}}00) = [0,\beta).
\end{eqnarray*}
For each $n \ge 1$, $B_n$ contains exactly two elements, one which has $b_0=2$ and one for which $b_0 =3$. So for fixed $b_0$, we can speak of the element $\Delta (b_0 \ldots b_{n-1})$ of $B_n$. We will define two sequences of sets $\{ R_{(2,n)} \}_{n \ge 1}$ and $\{ R_{(3,n)} \}_{n \ge 1}$, that represent the images of the elements of $B_n$ under $T^n$ and we will order them in two rows by assigning two extra parameters to each rectangle. Let
$$ R_0 = [0,2) \times [0,2) \times \{0\} \times \{ 0 \}$$
and for each $n \ge 1$, $j \in \{2,3\}$ define the sets
$$ R_{(j,n)} = T^n \Delta(j d^{(j)}_1 \ldots d^{(j)}_{n-1}) \times \Delta(\underbrace{0 \ldots 0}_{n \text{ times}}) \times \{ j \} \times \{ n \},$$
where the digits $d_n^{(j)}$ are the digits from the greedy expansions of $1$ and $\frac{1}{\beta^3}$ as given in (\ref{q:exp1}) and (\ref{q:exp2}). Then $R = R_0 \cup \bigcup_{n=1}^{\infty} (R_{(2,n)} \cup R_{(3,n)})$.
Let $\mathcal B_0$ denote the Borel $\sigma$-algebra on $R_0$ and on each of the rectangles $R_{(j,n)}$, let $\mathcal B_{(j,n)}$ denote the Borel $\sigma$-algebra defined on it. We can define a $\sigma$-algebra on $R$ as the disjoint union of all these $\sigma$-algebras,
$$ \mathcal B = \coprod_{j,n} \mathcal B_{(j,n)} \amalg \mathcal B_0.$$
Let $\bar \lambda$ be the measure on $(R, \mathcal B)$, given by the Lebesgue measure on each rectangle. Then $\bar \lambda (R) = 32 -14\beta$. If we set $\nu = \frac{1}{32-14\beta} \bar \lambda$, then $(R, \mathcal B, \nu)$ will be a probability space.\\
\indent The transformation $\T$ that we are going to define on this space will map $R_{(j,n)}$ onto $R_{(j,n+1)}$ if $\Delta (j d^{(j)}_1 \ldots d^{(j)}_{n-1}0)$ is non-full, otherwise a part of $R_{(j,n)}$ is mapped onto $R_{(j,n+1)}$ and the other part is mapped in $R_0$. We will define $\T$ piecewise on these sets.\\
On $R_0$, let
$$
\T (x,y,0,0) = \left\{
\begin{array}{ll}
(Tx, \frac{y}{\beta}, 0,0) \in R_0, & \mbox{if } x \in \Delta (0),\\
(Tx,\frac{y}{\beta},j,1) \in R_{(j,1)}, & \mbox{if } x \in \Delta (j), \; j \in \{2,3\}
\end{array}
\right.
$$
and for $(x,y,j,n) \in R_{(j,n)}$, let
$$
\T (x,y,j,n) = \left\{
\begin{array}{ll}
(Tx, y^{(j)},0,0) \in R_0, & \mbox{if } \Delta (j d^{(j)}_1 \ldots d^{(j)}_{n-1}0) \text{ is full}\\
&\; \text{and } x \in \Delta(0),\\
(Tx, \frac{y}{\beta},j,n+1)  \in R_{(j,n+1)}, & \mbox{if } \Delta (j d^{(j)}_1 \ldots d^{(j)}_{n-1}0)\\
&\; \text{is non-full or } x \not \in \Delta(0),\\
\end{array}
\right.
$$
where
$$
y^{(j)} =\displaystyle \frac{j}{\beta} + \frac{d_1^{(j)}}{\beta^2}+\frac{d_2^{(j)}}{\beta^3} + \ldots + \frac{d_{n-1}^{(j)}}{\beta^n} + \frac{y}{\beta}.
$$

\noindent Figure \ref{f:piles} shows the space $R$.
\begin{figure}[h]
\centering
\includegraphics[width=11cm]{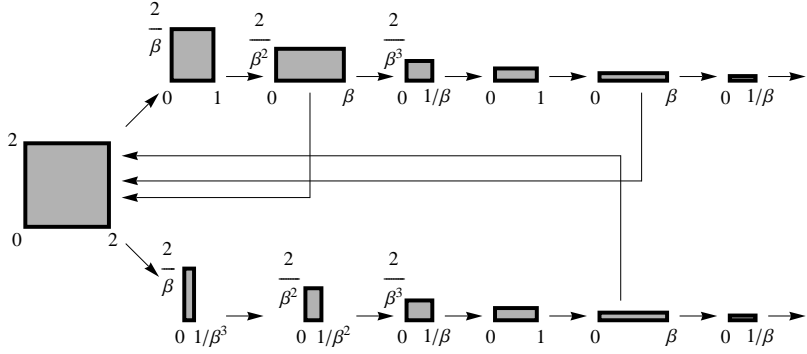}
\caption{The space $R$ consists of all these rectangles.}
\label{f:piles}
\centering
\end{figure}

\begin{remark}\label{r:bijective}
{\rm Notice that for $k\ge 1$, $\T$ maps all rectangles $R_{(2,n)}$ for which $n \neq 3k-1$ and all rectangles $R_{(3,n)}$ for which $n \neq 3k+2$ bijectively onto $R_{(2,n+1)}$ and $R_{(3,n+1)}$ respectively. The rectangles $R_{(2,3k-1)}$ and $R_{(3,3k+2)}$ are partly mapped onto $R_{(2,3k)}$ and $R_{(3,3k+3)}$ and partly into $R_0$. From Lemma \ref{l:generate} it follows that $\T$ is bijective.
}
\end{remark}

Let $\pi_1: R \to [0,2)$ be the projection onto the first coordinate. To show that $(R, \mathcal B, \nu, \T)$ is a version of the natural extension with $\pi_1$ as a factor map, we need to prove all of the following.
\begin{itemize}
\item[(i)] $\pi_1$ is a surjective, measurable and measure preserving map from $R$ to $[0,2)$.
\item[(ii)] For all $x \in R$, we have $(T \circ \pi_1)(x) = (\pi_1 \circ \T)(x)$.
\item[(iii)] $\T:R \to R$ is an invertible transformation.
\item[(iv)] $\mathcal B = \bigvee_{n=0}^{\infty} T^n \pi_1^{-1} (\mathcal B([0,2)))$, where $\bigvee_{n=0}^{\infty} T^n \pi_1^{-1} (\mathcal B([0,2)))$ is the smallest $\sigma$-algebra containing the $\sigma$-algebras $T^n \pi_1^{-1} (\mathcal B ([0,2)))$ for all $n \ge 1$.
\end{itemize}

It is clear that $\pi_1$ is surjective and measurable and that $T \circ \pi_1 = \pi_1 \circ \T$. Since $\T$ expands by a factor $\beta$ in the first coordinate and contracts by a factor $\beta$ in the second coordinate, it is also clear that $\T$ is invariant with respect to the measure $\nu$. Then $\mu = \nu \circ \pi_1^{-1}$ defines a $T$-invariant probability measure on $([0,2), \mathcal B([0,2)))$ and $\pi_1$ is measure preserving. This shows (i) and (ii). The invertibility of $\T$ follows from Remark \ref{r:bijective}, so that leaves only (iv). To prove (iv) we will have a closer look at the structure of the fundamental intervals and we will introduce some more notation.

\vskip .3cm

For a fundamental interval, $\Delta(b_0 \ldots b_q)$, the block of digits $b_0 \ldots b_q$ consists of several subblocks, each of which forms a full fundamental interval itself, except for possibly the last subblock. This last subblock will form a full fundamental interval if $\Delta(b_0 \ldots b_q)$ is full and it will form a non-full fundamental interval otherwise. We take these subblocks a small a possible, i.e. a new subblock starts, as soon as the previous subblock forms a full fundamental interval. Therefore, each of these subblocks consists only of the digit $0$ or is the beginning of the greedy expansion of $1$ or $\frac{1}{\beta^3}$, followed by the digit $0$, except possibly for the last subblock. For example, the block of digits from the fundamental interval $\Delta (2000300002002000)$ can be divided into the three subblocks, $200$, $0$ and $300002002000$. To make this subdivision more precise, we need the notion of return time. For points $(x,y) \in R_0$ define the {\it first return time to} $R_0$ by
$$r_1(x,y) = \min \{ n \ge 1: \T^n (x,y,0,0) \in R_0 \}$$
and for $k \ge 1$, let the $k${\it -th return time to} $R_0$ be given recursively by
$$r_k (x,y) =\min \{ n \ge r_{k-1} (x,y): \T^n (x,y,0,0) \in R_0\}.$$
Notice that this notion depends only on $x$, i.e. for all $y, y' \in R_0$ and all $k \ge 1$, $r_k(x,y)=r_k (x,y')$. So we can write $r_k (x)$ instead of $r_k (x,y)$. In this sense, for each $x \in [0,a_1)$ we can talk about the $k$-th return time of this element. If $\Delta (b_0 \ldots b_{q-1}) \in \Delta^{(q)}$, then for all $n \le q$, $\T^n$ maps the whole set $\Delta(b_0 \ldots b_{q-1}) \times [0,2) \times \{0\} \times \{0\} \subseteq R_0$ to the same rectangle in $R$. So the first several return times to $R_0$, $r_1, \ldots, r_{\kappa}$, are equal for all elements in $\Delta(b_0 \ldots b_{q-1})$. This means we can talk about the $k$-th return time to $R_0$ of this entire fundamental interval $\Delta(b_0 \ldots b_{q-1})$. Now suppose that $\Delta(b_0 \ldots b_{q-1}) \in \Delta^{(q)}$ is a full fundamental interval. Then there is a $\kappa \ge 1$ and there are numbers $r_i$, $1 \le i \le \kappa$ such that $r_i = r_i(x)$ for all $x \in \Delta (b_0 \ldots b_{q-1})$ and $r_{\kappa}=q$. Put $r_0=0$, then we can divide the block of digits $b_0 \ldots b_{q-1}$ into $\kappa$ subblocks $C_1, \ldots, C_{\kappa}$, where
$$ C_i = b_{r_{i-1}} \ldots b_{r_i-1}.$$
So $\Delta(b_0 \ldots b_{q-1}) = \Delta (C_1 \ldots C_{\kappa})$. These subblocks, $C_i$, have the following properties.
\begin{itemize}
\item[(i)] If $|C_i|$ denotes the length of block $C_i$, then $|C_i| = r_i-r_{i-1}$ for all $i \in \{ 1,2, \ldots, \kappa \}$.
\item[(ii)] If $b_{r_i}=0$, then $r_{i+1} = r_i+1$.
\item[(iii)] If $b_{r_i} = j \in \{ 2,3\}$, then the block $C_{i+1}$ is equal to $j$ followed by the first part of the greedy expansion of $1$ if $j=2$ and that of $1/ \beta^3$ if $j=3$. So $ C_{i+1} = j d^{(j)}_1 \ldots d^{(j)}_{|C_{i+1}|-1}$.
\item[(iv)] For all $i \in \{ 1, \ldots, \kappa \}$, $\Delta(C_i)$ is a full fundamental interval of rank $|C_i|$.
\end{itemize}
The above procedure gives for each full fundamental interval $\Delta(b_0 \ldots b_{q-1})$, a subdivision of the block of digits $b_0 \ldots b_{q-1}$ into subblocks $C_1, \ldots ,C_{\kappa}$, such that $\Delta(C_i)$ is a full fundamental interval of rank $| C_i |$ and $\Delta(b_0 \ldots b_{q-1})=\Delta(C_1 \ldots C_{\kappa})$. The next lemma is the last step in proving that the system $(R, \mathcal B, \nu, \T)$ is a version of the natural extension 

\begin{lemma}\label{l:bigvee}
The $\sigma$-algebra $\mathcal B$ on $R$ and the $\sigma$-algebra $\bigvee_{n=0}^{\infty} \T^n \pi_1^{-1}(\mathcal B([0,2)))$ are equal.
\end{lemma}

\begin{proof}
First notice that by Lemma \ref{l:generate}, each of the $\sigma$-algebras $\mathcal B_{(j,n)}$ is generated by the direct products of the full fundamental intervals, contained in the rectangle $R_{(j,n)}$. Also, $\mathcal B_0$ is generated by the direct products of the full fundamental intervals. It is clear that $\bigvee_{n=0}^{\infty} \T^n \pi_1^{-1}(\mathcal B([0,2))) \subseteq \mathcal B$. For the other inclusion, first take a generating rectangle in $R_0$:
$$\Delta(a_0 \ldots a_{p-1}) \times \Delta(b_0 \ldots b_{q-1}) \times \{0 \} \times \{0 \},$$
where $\Delta(a_0 \ldots a_{p-1})$ and $\Delta(b_0 \ldots b_{q-1})$ are full fundamental intervals. For the set $\Delta(b_0 \ldots b_{q-1})$ construct the subblocks $C_1, \ldots, C_{\kappa}$ as before. By Lemma \ref{l:full} $\Delta (C_{\kappa} C_{\kappa-1} \ldots C_1 a_0 \ldots a_{p-1})$ is a full fundamental interval of rank $p+q$. Then
$$ \pi_1^{-1} (\Delta (C_{\kappa} C_{\kappa-1} \ldots C_1 a_0 \ldots a_{p-1})) \cap R_0 \quad$$
$$ \quad = \Delta (C_{\kappa} C_{\kappa-1} \ldots C_1 a_0 \ldots a_{p-1}) \times [0,2) \times \{0 \} \times \{ 0 \}.$$
It is a well-known fact that for each full fundamental interval $\Delta(d_0 \ldots d_{n-1})$ and each $i \in \{ 1, \ldots, n-1 \}$, we have $ T^i \Delta (d_0 \ldots d_{n-1}) = \Delta(d_i \ldots d_{n-1})$. This, together with the definitions of the blocks $C_i$ and the transformation $\T$ leads to
$$ \T^q ( \pi_1^{-1} (\Delta (C_{\kappa} C_{\kappa-1} \ldots C_1 a_0 \ldots a_{p-1})) \cap R_0) \quad \quad$$
$$ \quad =  \Delta (a_0 \ldots a_{p-1}) \times \Delta (C_1 C_2 \ldots C_{\kappa}) \times \{ 0 \} \times \{ 0 \}.$$
So
$$ \Delta(a_0 \ldots a_{p-1}) \times \Delta(b_0 \ldots b_{q-1}) \times \{0 \} \times \{0 \} \subseteq \bigvee_{n=0}^{\infty} \T^n \pi_1^{-1}(\mathcal B([0,2))).$$
Now let $\Delta(a_0 \ldots a_{p-1}) \times \Delta(b_0 \ldots b_{q-1}) \times \{j \} \times \{n \}$
be a generating rectangle for $\mathcal B_{(j,n)}$, for $j \in \{2,3\}$ and $n \ge 1$. So $\Delta(a_0 \ldots a_{p-1})$ and $\Delta(b_0 \ldots b_{q-1})$ are again full fundamental intervals. Notice that 
$$\Delta(b_0 \ldots b_{q-1}) \subseteq \Delta(\underbrace{0 \ldots 0}_{n \text{ times}}),$$
which means that $q \ge n$. Also $b_i =0$ and thus $r_{i+1}=i+1$ for all $i \in \{ 0, \ldots, n-1\}$. So, if we divide $b_0 \dots b_{q-1}$ into subblocks $C_i$ as before, we get that $C_1 = C_2 = \ldots =C_n =0$, that $\kappa \ge n$ and that $|C_{n+1}| + \ldots + |C_{\kappa}|=q-n$. Consider the set 
$$ C = \Delta (C_{\kappa} C_{\kappa-1} \ldots C_{n+1} j d^{(j)}_1 \ldots d_{n-1}^{(j)}a_0 \ldots a_{p-1}).
$$
We will show the following.\\
Claim: The set $C$ is a fundamental interval of rank $p+q$ and $T^q C=\Delta(a_0 \ldots a_{p-1})$.

\vskip .2cm

\noindent First notice that
$$ C= \Delta (C_{\kappa} C_{\kappa-1} \ldots C_{n+1}) \cap T^{n-q}  \Delta(j d^{(j)}_1 \ldots d_{n-1}^{(j)}) \cap T^{-q}\Delta(a_0 \ldots a_{p-1}).$$
So obviously,
$$
T^q C \subseteq T^q \Delta (C_{\kappa} C_{\kappa-1} \ldots C_{n+1}) \cap T^n  \Delta(j d^{(j)}_1 \ldots d_{n-1}^{(j)}) \cap \Delta(a_0 \ldots a_{p-1}).
$$
By Lemma \ref{l:full}, $\Delta (C_{\kappa} C_{\kappa-1} \ldots C_{n+1})$ is a full fundamental interval of rank $q-n$, so $ T^q \Delta (C_{\kappa} C_{\kappa-1} \ldots C_{n+1}) = [0,2)$. Now, by the definition of $R_{(j,n)}$ we have that
\begin{equation}\label{q:deltaa}
\Delta(a_0 \ldots a_{p-1}) \subseteq T^n  \Delta(j d^{(j)}_1 \ldots d_{n-1}^{(j)}),
\end{equation}
and thus $T^q C \subseteq   \Delta(a_0 \ldots a_{p-1})$.\\
For the other inclusion, let $z \in \Delta (a_0 \ldots a_{p-1})$. By (\ref{q:deltaa}),
there is an element $y$ in $\Delta(j d^{(j)}_1 \ldots d_{n-1}^{(j)})$, such that $T^n y =z$. And since $ T^{q-n} \Delta (C_{\kappa} C_{\kappa-1} \ldots C_{n+1}) = [0,2)$, there is an $x \in \Delta (C_{\kappa} C_{\kappa-1} \ldots C_{n+1})$ with $T^{q-n}x = y$, so $T^q x = z$. This means that
$$ z \in T^q \Delta (C_{\kappa} C_{\kappa-1} \ldots C_{n+1}) \cap T^n \Delta(j d^{(j)}_1 \ldots d_{n-1}^{(j)}) \cap \Delta(a_0 \ldots a_{p-1}).$$ 
So $ T^q C= \Delta(a_0 \ldots a_{p-1})$ and this proves the claim.

\vskip .2cm

\noindent Consider the set $D= \pi_1^{-1} (C) \cap R_0$. Then as before, we have
$$\T^{q-n} D = \Delta (j d^{(j)}_1 \ldots d^{(j)}_{n-1}a_0 \ldots a_{p-1}) \times \Delta(C_{n+1} C_{n+2} \ldots C_{\kappa}) \times \{0 \} \times \{0 \}.$$
And after $n$ more steps,
\begin{eqnarray*}
\T^q D &=& \Delta(a_0 \ldots a_{p-1}) \times \Delta (\underbrace{00\ldots 0}_{n \text{ times}} C_{n+1} \ldots C_{\kappa}) \times \{ j \} \times \{ n \}\\
&=& \Delta(a_0 \ldots a_{p-1}) \times \Delta(b_0 \ldots b_{q-1}) \times \{j \} \times \{n \}.
\end{eqnarray*}
So, 
$$\Delta(a_0 \ldots a_{p-1}) \times \Delta(b_0 \ldots b_{q-1}) \times \{j \} \times \{n \} \in \bigvee_{n=0}^{\infty} \T^n \pi_1^{-1}(\mathcal B([0,2)))$$
and thus we see that
$$ \mathcal B = \bigvee_{n=0}^{\infty} \T^n \pi_1^{-1}(\mathcal B([0,2))).  \qedhere$$
\end{proof}

\noindent This leads to the following theorem.

\begin{theorem}
The dynamical system $(R, \mathcal B, \nu, \T)$ is a version of the natural extension of the dynamical system $([0,2), \mathcal B([0,2)), \mu, T)$, where $\mu = \nu \circ \pi_1^{-1}$ is an invariant probability measure of $T$, equivalent to the Lebesgue measure on $[0,2)$, whose density function, $h:[0,2) \to [0,2)$, is given by
\begin{eqnarray*} 
h(x) &=&  \frac{1}{16-7\beta} [(1+2\beta)1_{[0, 1/ \beta^3)}(x) + (2+\beta)1_{[1/ 
\beta^3, 1/ \beta^2)}(x)\\
&& + 2\beta 1_{[1/ \beta^2, 1/ \beta)}(x) + \beta^2 1_{[1/ \beta, 1)}(x) + \beta 1_{[1, \beta)}(x) + 1_{[\beta, 2)}(x)].
\end{eqnarray*}
\end{theorem}

\begin{proof}
The proof follows from Remark \ref{r:generate}, the properties of $\pi_1$ and Lemma \ref{l:bigvee}.
\end{proof}

\section{Towering the orbits}
For the second version of the natural extension, we will define a transformation on a certain subset of $[0,2)\times [0,2\beta)$, using the transformation $\T$, defined in the previous section. Define for $n \ge 1$ the following intervals:
$$ I_{(2,n)} = \left[\frac{2}{\beta^2} + \frac{2}{\beta^2}\sum_{j=1}^{n-1} \frac{1}{\beta^j}, \frac{2}{\beta^2} + \frac{2}{\beta^2}\sum_{j=1}^{n} \frac{1}{\beta^j} \right)$$
and
$$ I_{(3,n)} = \left[2 + \frac{2}{\beta^2}\sum_{j=1}^{n-1} \frac{1}{\beta^j}, 2 + \frac{2}{\beta^2}\sum_{j=1}^{n} \frac{1}{\beta^j}\right),$$
where $\displaystyle \sum_{j=1}^{0} \frac{1}{\beta^j} =0$. Let $I_0 = [0,\frac{2}{\beta^2})$. Notice that all of these rectangles are disjoint and that $\displaystyle \bigcup_{n=1}^{\infty} I_{(2,n)} = \left[ \frac{2}{\beta^2}, 2 \right)$ and $\displaystyle \bigcup_{n=1}^{\infty} I_{(3,n)} = [2, 2 \beta)$, so that these intervals together with $I_0$ form a partition of $[0,2\beta)$. Now define the subset $I \subseteq [0,2)\times [0,2\beta)$ by
$$I = ([0,2) \times I_0) \cup \bigcup_{n=1}^{\infty} (([0, T^{n-1} 1) \times I_{(2,n)}) \cup ([0, T^{n-1} \frac{1}{\beta^3}) \times I_{(3,n)})) $$
and let the function $\phi:I \to R$ be given by
$$\phi(x,y) = \left\{
\begin{array}{ll}
(x, \beta^2 (y-\frac{2}{\beta^2}-\frac{2}{\beta^3}\sum_{j=0}^{n-1} \frac{1}{\beta^j}), 2,n), & \text{if } y \in I_{(2,n)},\\
(x, \beta^2 (y-2-\frac{2}{\beta^3}\sum_{j=0}^{n-1}\frac{1}{\beta^j}), 3,n), & \text{if } y \in I_{(3,n)},\\
(x, \beta^2 y,0,0), & \text{if } y \in I_0.\\
\end{array}
\right.$$
So $\phi$ maps $I_0$ to $R_0$ and for all $n \ge 1$, $j \in \{ 2,3 \}$, $\phi$ maps $I_{(j,n)}$ to $R_{(j,n)}$. Clearly, $\phi$ is a measurable bijection. Define the transformation $\tilde {\T}: I \to I$, by 
$$\tilde {\T} (x,y) = \phi^{-1} (\T (\phi(x,y))).$$
It is straightforward to check that $\tilde{\T}$ is invertible. In Figure \ref{f:tower} we see this transformation.
\begin{figure}[h]
\begin{center}
\includegraphics[width=12cm]{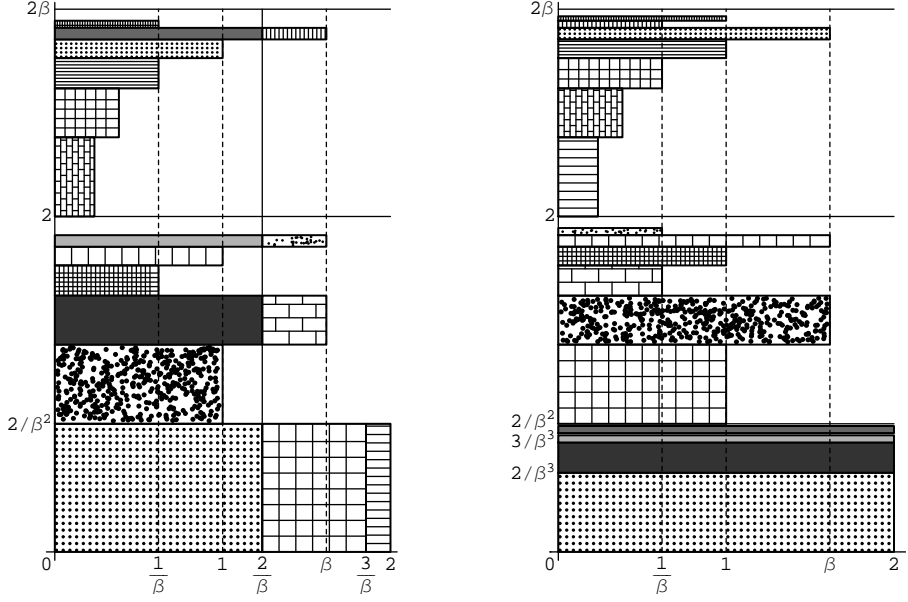}
\caption{The transformation $\tilde{\T}$ maps the regions on the left to the regions on the right.}
\label{f:tower}
\end{center}
\end{figure}
Let $\mathcal I$ be the collection of Borel sets on $I$. If $\lambda_2$ is the 2-dimensional Lebesgue measure, then
$$ \lambda_2 (I) = 78-46 \beta = \frac{1}{\beta^2} \bar \lambda(R).$$
Define a measure $\tilde{\nu}$ on $(I, \mathcal I)$ by setting $\tilde{\nu} (E) = (\nu \circ \phi)(E)$, for all $E \in \mathcal I$. Then $\phi$ is measure preserving and the systems $(R, \mathcal B, \mu, \T)$ and $(I, \mathcal I, \tilde{\nu}, \tilde{\T})$ are isomorphic. Notice that $\tilde{\nu}$ is the normalized 2-dimensional Lebesgue measure on $(I, \mathcal I)$ and that the projection of $\tilde{\nu}$ on the first coordinate gives $\mu$ again.
The following lemma is now enough to show that $(I, \mathcal I, \tilde{\nu}, \tilde{\T})$ is a version of the natural extension of $([0,2), \mathcal B([0,2)), \mu, T)$.

\begin{lemma}
The $\sigma$-algebras $\mathcal I$ and $\bigvee_{n=0}^{\infty} \tilde{\T}^n \pi_1^{-1} (\mathcal B([0,2))) $
are equal.
\end{lemma}

\begin{proof}
It is easy to see that $\bigvee_{n=0}^{\infty} \tilde{\T}^n \pi_1^{-1} (\mathcal B([0,2))) \subseteq \mathcal I$.
For the other inclusion, notice that the direct products of full fundamental intervals contained in
$$([0,2) \times I_0) \cup \bigcup_{n=1}^{\infty} ( [0, T^{n-1} 1) \times I_{(2,n)}),$$
generate the restriction of $\mathcal I$ to this set. If $\Delta(b_0 \ldots b_{n-1}) \in \Delta^{(n)}$ is full in $[0, \frac{2}{\beta})$, then the set $2+ \Delta(b_0 \ldots b_{n-1})$ is a subset of $[2, 2\beta)$. So the direct products of full fundamental intervals in $[0, \beta)$ and sets of the form $2+ \Delta(b_0 \ldots b_{n-1})$ contained in $\bigcup_{n=1}^{\infty} ([0, T^{n-1} \frac{1}{\beta^3}) \times I_{(3,n)})$, generate the restriction of $\mathcal I$ to this set. Since $\tilde{\T}$ is isomorphic to $\T$, the fact that
$$\mathcal I \subseteq \bigvee_{n=0}^{\infty} \tilde{\T}^n \pi_1^{-1} (\mathcal B([0,2)))$$
now can be proven in a way similar to the proof of Lemma \ref{l:bigvee}.
\end{proof}

\section{Concluding remark}

In the previous sections we have defined two dynamical systems that are versions of the natural extension of the dynamical system $([0,2), \mathcal B([0,2)), \mu, T)$, where $T$ is the greedy $\beta$-transformation with deleted digits for $\beta = \frac{1+\sqrt 5}{2}$ and $A=\{ 0,2,3\}$. This gave us the possibility to find the density function of the invariant measure of $T$, equivalent to the Lebesgue measure on $[0,2)$. An important feature of the transformation $T$, that was used in both versions is that the orbits of the points $1$ and $\frac{1}{\beta^3}$ and the interval $\Delta(3)$ are disjoint. If this would not be the case, defining a version of the natural extension of a greedy $\beta$-transformation with three deleted digits would require extra effort. It is probably the first version of the natural extension that can be adapted to this most easily.

\end{document}